\documentclass{article}
\usepackage[utf8]{inputenc}
\usepackage{tikz}
\usepackage{graphicx} 

\title{On The Random Tur\'an number of linear cycles}

\author
{Dhruv Mubayi\thanks{Department of Mathematics, Statistics, and Computer Science, University of Illinois, Chicago, IL, 60607 USA. Email: mubayi@uic.edu.
Research partially supported by NSF Awards DMS-1763317,
DMS-1952767 and DMS-2153576, by a Humboldt Research Award, and by a Simons Fellowship.}
\and  Liana Yepremyan\thanks{
    Department of Mathmatics, Emory University, 
    Atlanta, Georgia, 30322, USA. Email: \texttt{lyeprem}@\texttt{emory.edu}.
    }
}
\date{\today}

\usepackage{hyperref}
\usepackage{float}
\usepackage{amsmath, color, amssymb, amsfonts, amsthm, enumerate, a4wide}

\usepackage{cleveref}

\newcommand{\eps}{\ensuremath{\varepsilon}}

\newtheorem{theorem}{Theorem}[section] 

\newtheorem{claim}[theorem]{\bf Claim}
\newtheorem{lemma}[theorem]{\bf Lemma}
\newtheorem{conjecture}[theorem]{\bf Conjecture}

\newtheorem{proposition}[theorem]{\bf Proposition}

\newcommand\numberthis{\addtocounter{equation}{1}\tag{\theequation}}
\newcommand\numeq[1]%
  {\stackrel{\scriptscriptstyle(\mkern-1.5mu#1\mkern-1.5mu)}{=}}
  
\newcommand{\Prob}{\mathbb{P}}
\newcommand{\E}{\mathbb{E}}
\newcommand{\ex}{{\rm ex}}
\newcommand{\Ext}{{\rm Ext}}
\newcommand{\plog}{{\rm plog}}

\begin{document}
\setlength{\baselineskip}{12pt}

\maketitle
\begin{abstract}  
Given two $r$-uniform hypergraphs $G$ and $H$ the Tur\'an number $\ex(G, H)$   is the  maximum number of edges in an  $H$-free subgraph of $G$.   We study the typical value of  $\ex(G, H)$ when $G=G_{n,p}^{(r)}$, the Erd\H{o}s-R\'enyi random $r$-uniform hypergraph, and $H=C_{2\ell}^{(r)}$, the $r$-uniform linear cycle of length $2\ell$. The case of graphs ($r=2$) is a longstanding open problem that has been investigated by many researchers. We  determine the order of magnitude of  $\ex\left(G_{n,p}^{(r)}, C_{2\ell}^{(r)}\right)$ for  all $r\geq 4$  and all $\ell\geq 2$ up to polylogarithmic factors for all values of $p=p(n)$.
 Our proof is based on the container method and uses a balanced supersaturation result for linear even cycles which  improves upon previous such results by Ferber-Mckinley-Samotij  and Balogh-Narayanan-Skokan.
\end{abstract}
\begin{section}{Introduction}
 Given two  $r$-uniform hypergraphs $G$ and $H$ (henceforth $r$-graphs), $G$ is $F$-free if $G$ contains no (not necessarily induced) subgraph isomorphic to $H$.  The Tur\'an number $\ex(G, H)$  is the   maximum of $e(G')$ over all $H$-free subgraphs of $G' \subseteq G$. When $G=K_{n}^{(r)}$, the complete $r$-graph on $n$ vertices, $\ex(G, H)$ is simply denoted by $\ex(n,H)$. Determining $\ex(n,H)$  and its order of magnitude for large $n$ is a central problem in extremal graph theory, known as the Tur\'an problem of $H$. For more on Tur\'an numbers, we refer the reader to the excellent surveys~\cite{FS} for graphs and~\cite{keevash} for hypergraphs.

A hypergraph is linear if every two edges have at most one vertex in common. In this paper we study $\ex(G,H)$ when $G$ is the random $r$-graph $G_{n,p}^{(r)}$, and $H$ is a linear even $r$-uniform cycle. Here $G_{n,p}^{(r)}$ is the $r$-graph  on $n$ labelled vertices whose edges are independently present with probability $p=p(n)$.  The random variable $\ex(G_{n,p},H) = \ex(G_{n,p}^{(2)}, H)$ was first considered by
Babai, Simonovits, and Spencer~\cite{BSS} who treated the case in which $H$ has chromatic
number three and $p$ is a constant. The systematic study of $\ex(G_{n,p},H)$  was initiated by Kohayakawa, Luczak and R\"{o}dl~\cite{KLR} (see the survey~\cite{RS} for more extremal results in random graphs). One of their  conjectures, resolved independently by Conlon
and Gowers~\cite{CG} and by Schacht~\cite{S}, determines the asymptotic value of $\ex(G_{n,p}, H)$ whenever $H$
has chromatic number at least three.

The behaviour of $\ex(G_{n,p}, H)$ when $H$ is bipartite is a wide open problem that is closely  related to  the order of magnitude of the usual Tur\'an numbers $\ex(n, H)$. One case of bipartite $H$ that has been extensively studied is when $H=C_{2\ell}$, the even cycle on $2\ell$ vertices. Haxell, Kohayakawa and \L{}uczak~\cite{HKL} determined the so-called \emph{threshold}  $p$ for $H=C_{2\ell}$.  Namely they showed that a.a.s. if $p\gg n^{-1+1/(2\ell-1)} $ then $\ex(G_{n,p}, C_{2\ell})\ll e(G_{n,p})$, and if $p\ll n^{-1+1/(2\ell-1)}$  then  $\ex(G_{n,p},C_{2\ell})=(1-o(1))e(G_{n,p})$. Kohayakawa, Kreuter and Steger~\cite{KKS} improved on the second part and obtained more precise bounds for a certain range of $p$.
Finally, using the powerful hypergraph container method invented independently by Balogh, Morris and Samotij~\cite{BMS} and  Saxton and Thomasson~\cite{ST}, Morris and Saxton~\cite{MS} further improved the upper bounds on  $\ex(G_{n,p}, C_{2\ell})$ for a broader range of $p$.

\begin{theorem} [\cite{HKL, KKS, MS}] 
\label{thm:graphcycles}For every $\ell\geq 2$, there exists  $C=C(\ell)$ such that a.a.s.
$$ \ex\left(G_{n,p}, C_{2\ell}\right) \leq \begin{cases} n^{1+1/(2\ell-1)} (\log{n})^{2},  &   \mbox{if } p \leq n^{-(\ell-1)/(2\ell-1)}(\log{n})^{2\ell} ,\\
Cp^{1/\ell}n^{1+1/\ell},& \mbox{otherwise.} \end{cases}$$

\end{theorem}
In Theorem~\ref{thm:graphcycles}, the first bound is sharp up to a polylog factor by the results of \cite{KKS}. As for the second bound, an old conjecture of Erd\H{o}s and Simonovits says that  there is a graph of girth at least $2\ell+1$ and  $\Omega(n^{1+1/\ell})$ edges. If this conjecture is true, then the second bound is also  sharp up to the value of the constant $C$ (see the discussion after Conjecture 2.3 in~\cite{MS}).

The problem of finding the largest $H$-free subgraph of $G_{n,p}^{(r)}$ is closely related to the problem of determining $|\mathrm{Forb}(n,H)|$, the number of $H$-free subgraphs on $n$ labelled vertices.   There is a well-developed theory for the latter problem for  $r$-graphs that are not $r$-partite~\cite{EFR, NRS, NR}. 
The corresponding question for $r$-partite $r$-graphs was initiated in a recent paper of the first author and Wang~\cite{MW}. The $r$-graph $C_k^{(r)}$ is obtained from the $2$-graph $k$-cycle $C_k$  by adding $r-2$ new vertices of degree one to each graph edge (thus enlarging each graph edge to an $r$-graph edge).
The authors in~\cite{MW} determined the asymptotics  of $|\mathrm{Forb}(n, C_{k}^{(r)})|$ for even $k$ and $r=3$, and conjectured that similar results hold for all $k,r\ge 3$. This was later confirmed by Balogh, Narayanan and Skokan~\cite{BNS}.   Soon after, Ferber, McKinley and Samotij~\cite{FMS} proved similar results for a  much larger class of $r$-graphs that includes linear cycles and linear paths. In the last few years  it is quite standard to use the container method to obtain the number of $H$-free graphs. However, to obtain precise bounds on  $\ex(G_{n,p}^{(r)},H)$, one needs a more delicate count on the number of $H$-free graphs with certain number of edges.  This is exactly how Morris and Saxton~\cite{MS} obtained their bounds on  $\ex(G_{n,p}, C_{2\ell})$.  These bounds usually rely on so-called balanced supersaturaion results. A  typical supersaturation result says that if an $r$-graph $G$ has more than $\ex(n,H)$  edges, then there are many copies of $H$ in $G$. An optimal supersation result would say that there are as many copies of $H$ in $G$ as one would expect in a random graph with the same edge-density. The balanced supersaturation typically requires to obtain a collection of such  copies of $H$ such that most tuples of edges  are in a bounded number  of copies of  $H$ in the collection.  Unfortunately the results of~\cite{BNS, FMS} both rely on supersaturation theorems that are not strong enough to imply  anything nontrivial for  $\ex\left(G_{n,p}^{(r)}, C_{2\ell}^{(r)}\right)$ when $\ell \geq 2$ and $p=o(1)$.  In this paper, we prove a stronger supersaturation result that does this. The bounds in our supersaturation result  depends on  the largest shadow of the underlying hypergraph and codegrees of pairs of vertices in the shadow graph. We discuss the proof method more at the beginning of Section~\ref{proofs}.

\section{Notation}

Given functions $f,g: \mathbb R^+ \to \mathbb R^+$, we write $f(n)\ll g(n)$ to mean $f(n)/g(n)\rightarrow 0$ as $n\rightarrow \infty$, $f(n) = O(g(n))$ to mean that there is an absolute positive constant $C$ such that $f(n) < C g(n)$, $f(n)=\Omega(n)$ to mean that $g(n) = O(f(n))$ and $f(n) = \Theta(g(n))$ to mean that $f(n)=O(g(n))$ and $f(n)= \Omega(n)$. 
 Throughout this paper, we say that a statement depending on $n$ holds asymptotically almost surely (abbreviated a.a.s.) if the probability that it holds tends to $1$ as $n$ tends to infinity. We will also use a notation $\plog_{\ell,r}(n)$ to denote any function which is of order $\Theta((\log{n})^{f(\ell,r)})$ where the power of the logarithm, $f(\ell,r)$ depends on $\ell$ and $r$ and is positive.  Since in this paper we are only interested in determining the  order of magnitude of $\ex(G_{n,p}^{(r)}, C_{2\ell}^{(r})$ up to polylogarithmic  factors, we use this convention to make the paper easier to read and make no attempt to optimise the logarithmic terms. So  when we write a statement holds using $\plog_{\ell,r}(n)$  it means there exists a function of $f(\ell,r)$ for which the statement is true.
 
For an $r$-graph $G$, we write $e(G)$ for its number of edges and $d(G)=r\cdot e(G)/|V(G)|$ for its average degree. If $G' \subset G$, the $r$-graph $G-G'$ has vertex set $V(G)$ and edge set $E(G)\setminus E(G')$. In particular, $e(G-G') = |E(G)\setminus E(G')|$. For an $r$-partite $r$-graph $G$ with vertex partition $(V_1, V_2, \dots, V_r)$, for $1\leq i <j \leq r$ we write $\partial_{V_i,V_j}(G)$ for the pairs $(v_i,v_j)$ with $v_i\in V_i$ and $v_j\in V_j$ such that there exists some $e\in E(G)$ with $\{v_i,v_j\}\subseteq e$. For any $r$-graph $G$,  $1\leq j\leq r$ and  any $j$-tuple $\sigma$, the \emph{degree} of $\sigma$, written $d_{G}(\sigma)$ is the number of edges that contain $\sigma$; when $\sigma=\{u,v\}$ we simplify the notation to $d_G(u,v)$. We denote by $\Delta_j(G)$ the maximum $d_G(\sigma)$ among all $j$-tuples $\sigma$.  Whenever $G$ is clear from the context we will drop it from the notation.  Given an $r$-graph $G$ and $0<\tau<1$, the \emph{co-degree  function} $\delta(G,\tau)$  is 
$$\delta(G,\tau)=\frac{1}{d(G)}\sum_{j=2}^{r}{\frac{\Delta_{j}}{\tau^{j-1}}}.$$

 A vertex subset $I$ is called \emph{independent} in $G$ if there is no edge $e\in E(G)$ such that $e\subseteq I$. For  $A\subseteq V(G)$ we define $G[A]$, \emph{the subgraph induced by} $A$, to be the $r$-graph with vertex set $A$ and edge set comprising all those  $e\in E(G)$ for which $e\subseteq A$.

\section{Main result}

 Our main result is the following extension of 
Theorem~\ref{thm:graphcycles} to linear even  cycles $C_{2\ell}^{(r)}$ for $r\geq 4$. In all of the below results all $o(1)$ error terms in the exponents are of order $O(\log\log n /\log n)$.

\begin{theorem}\label{thm:mainforratleast4}For every $\ell\geq 2$ and $r\geq 4$ a.a.s. for all $p\geq n^{-(r-2)+\frac{1}{2\ell-1}+o(1)},$
$$ \ex\left(G_{n,p}^{(r)}, C_{2\ell}^{(r)}\right) \leq 
pn^{r-1+o(1)}.$$
\end{theorem}

The upper bounds in Theorem~\ref{thm:mainforratleast4} together with monotonicity of the function $\ex(G_{n,p}^{(r)}, H)$ in $p$  give us the order of magnitude of $\ex\left(G_{n,p}^{(r)}, C_{2\ell}^{(r)}\right)$ up to  polylogarithmic factors.

\begin{theorem}\label{thm:ourcor}For every $\ell\geq 2$ and $r\geq 4$ a.a.s. the following holds: 
$$ \ex\left(G_{n,p}^{(r)}, C_{2\ell}^{(r)}\right) = \begin{cases} 
(1-o(1))pn^r,  & \mbox{if } {n^{-r} \ll p\ll n^{-(r-1)+\frac{1}{2\ell-1}}} \\ 
n^{1+\frac{1}{2\ell-1}+o(1)}, & \mbox{if } {n^{-(r-1)+\frac{1}{2\ell-1}+o(1)}\leq p\leq  n^{-(r-2)+\frac{1}{2\ell-1}+o(1)}} \\ 
pn^{r-1+o(1)}, & \mbox{if } { p \geq n^{-(r-2)+\frac{1}{2\ell-1}+o(1)}}. \end{cases}$$
\end{theorem}

\begin{proof} When $n^{-r}\ll p \ll {n^{-(r-1)+\frac{1}{2\ell-1}}}$ a.a.s. $G_{n,p}^{(r)}$ has a $C_{2\ell}^{(r)}$-free subgraph with $(1+o(1))e(G_{n,p}^{(r)})$  edges by a simple deletion argument (see Proposition~\ref{prop:lowerbound1} for details), and this is the best possible, therefore in this regime  $\ex\left(G_{n,p}^{(r)}, C_{2\ell}^{(r)}\right) =\Theta(pn^r)$. For $p\gg n^{-(r-1)}$, $G_{n,p}^{(r)}$ a.a.s contains $\Omega(pn^{r-1})$ edges  containing a fixed vertex (see Proposition~\ref{prop:lowerbound2} for the proof).  This together with the bound in Theorem~\ref{thm:mainforratleast4} determines the optimal behaviour of   $\ex\left(G_{n,p}^{(r)}, C_{2\ell}^{(r)}\right)$ up to $n^{o(1)}$ factors for $p\geq n^{-(r-2)+\frac{1}{2\ell-1}+o(1)}$, that is, in this range  $\ex\left(G_{n,p}^{(r)}, C_{2\ell}^{(r)}\right)=  pn^{r-1+o(1)}$. To understand the behaviour of  $\ex\left(G_{n,p}^{(r)}, C_{2\ell}^{(r)}\right)$  when $n^{-(r-1)+\frac{1}{2\ell-1}+o(1)}\leq p \leq n^{-(r-2)+\frac{1}{2\ell-1}+o(1)}$   we need to combine our previous two bounds together with monotonicity of the  $H$-free property. Indeed, $\ex(G^{(r)}_{n,p}, H) \geq \ex(G^{(r)}_{n,q}, H)$ for $p \ge q$ since if $G^{(r)}_{n,q}$ has an $H$-free subgraph with $m$ edges, then $G^{(r)}_{n,p}$ has an $H$-free subgraph with $m$ edges. Therefore, if we let $p_0=n^{-(r-1)+\frac{1}{2\ell-1}+o(1)}$ then by  our earlier discussion we know that 
$$\ex\left(G_{n,p_0}^{(r)}, C_{2\ell}^{(r)}\right) = (1-o(1))p_0n^{r}= n^{1+\frac{1}{2\ell-1}+o(1)}.$$
If we let $p_1= n^{-(r-2)+\frac{1}{2\ell-1}+o(1)}$, then $\ex\left(G_{n,p_1}^{(r)},C_{2\ell}^{(r)}\right)\leq n^{1+\frac{1}{2\ell-1}+o(1)}$ by the bound in Theorem~\ref{thm:mainforratleast4}. Thus by monotonicity, $\ex\left(G_{n,p}^{(r)}, C_{2\ell}^{(r)}\right)= n^{1+\frac{1}{2\ell-1}+o(1)}$ for $  p\in[ p_0, p_1]$.
\end{proof}

\begin{figure}[H]
\centering
\begin{tikzpicture}
    \def\xmax{10}
    \def\ymax{5}
    \def\startx{1}
    \def\starty{0}
    \draw[->] (\startx,\starty) -- (\startx+\xmax+1,\starty) node[right] {$p$};
    \draw[->] (\startx,\starty) -- (\startx,\starty+\ymax+1) node[above] {$\ex(n,G_{n,p}^{(r)})$};
    \draw[blue, thick] (\startx+0.5,\starty)  -- (\startx+2,2) -- (\startx+8,2) -- (\startx+\xmax,\ymax);
    \node[above] at (2,0.7) {$pn^r$};
    \node[above] at (10,4) {$pn^{r-1}$};
    \node[below left] at (\startx,\starty) {$0$};
    \node[below] at (\startx+0.5,\starty) {$n^{-r}$};
    \node[below] at (\startx+2,\starty) {$n^{-(r-1)+\frac{1}{2\ell-1}}$};
    \node[below] at (\startx+8,\starty) {$n^{-(r-2)+\frac{1}{2\ell-1}}$};
    \node[below] at (\startx+\xmax,\starty) {$1$};
    \node[right] at (\startx,\starty+2.5) {$n^{1+\frac{1}{2\ell-1}}$};
    \draw[dashed] (\startx+2,0) -- (\startx+2,2);
    \draw[dashed] (\startx+8,0) -- (\startx+8,2);
    \draw[dashed] (\startx+\xmax,0) -- (\startx+\xmax,\ymax);
    \draw[dashed] (\startx,2) -- (\startx+2,2);
    \draw[dashed] (\startx,4) -- (\startx,\starty+4);
    \draw[dashed] (\startx,\starty+2) -- (\startx+6,\starty+2);
\end{tikzpicture}
  \caption{The behaviour of  $\ex\left(G_{n,p}^{(r)}, C_{2\ell}^{(r)}\right)$ for $r\geq 4$, $\ell\geq 2$ up to polylogarithmic terms} 
  \label{fig:general}
  \end{figure}
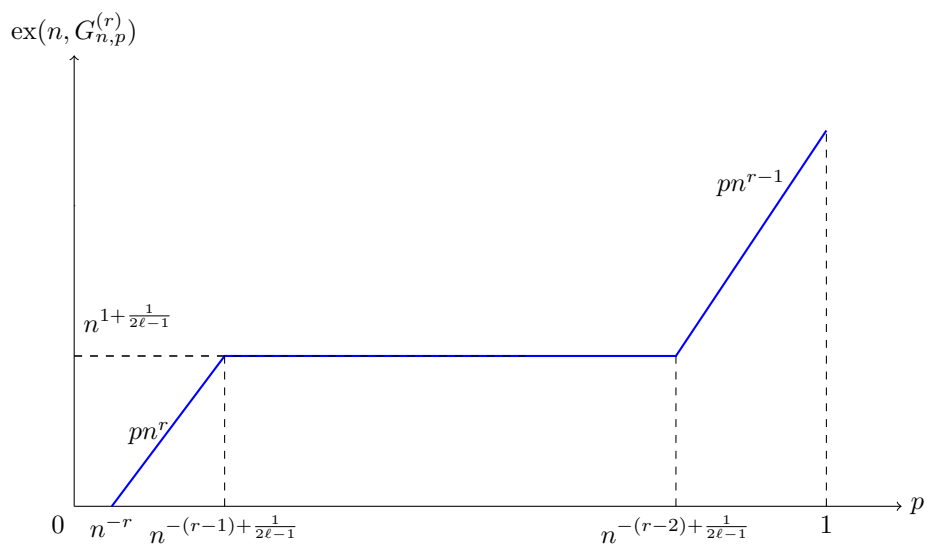

The question for odd linear cycles $C_{2\ell+1}^{(r)}$ for $\ell\geq 2$ or $\ell=1$ and $r\geq 4$ remains wide open. We believe that the behaviour for odd cycles is very different compared to our results.

\end{section}

\begin{section}{Tools}

We use standard Chernoff bounds, a version of the hypergraph container theorem and a balanced supersaturation theorem for even cycles in $2$-graphs  by Morris and Saxton. The latter is used as a black box to obtain a similar balanced supersaturation result for $r$-uniform cycles, $r\geq 3$. 

\begin{lemma}[\cite{molloy2013graph}, Chernoff bound] Given a binomially distributed variable $X\in \mathrm{Bin}(n, p)$  and  $0<a\leq 3/2$,
 $$\Prob{[|X-\E[X]|\geq a \E[X]]}\leq 2e^{-\frac{a^2}{3}\E[X]}.$$
\end{lemma}

\begin{theorem}[\cite{MS}, Theorem 4.2] \label{thm:MSContainers}For each $r\geq 2$ there exist   $\eps_0=\eps_0(r)$ such that for all $0<\eps <\eps_0$ the following holds. Let $G$ be an $r$-graph with $N$ vertices  and suppose  $\delta(G, \tau)\leq \eps $, for some $0<\tau < 1/2$. Then there exists a collection of $\mathcal{C}\subset  \mathcal{P}(V(G))$ of at most $$\exp\left(\frac{\tau N\log{1/\tau}}{\eps}\right)$$  subsets of $V(G)$ (called the containers of $G$) such that
\begin{itemize}
\item [(1)] for each independent set $I\subset V(G)$ there exists a container $C\in \mathcal{C}$ such that  $I\subseteq C$,
\item [(2)] $e(G[C])\leq (1-\eps) e(G)$, for each container $C\in \mathcal{C}$.
\end{itemize}
\end{theorem}
Below we consider the collection of copies of $H$ in $G$. This can be viewed as an $e(H)$-graph with vertex set $E(G)$, and   edge set comprising collection of edges of $G$ that form a copy of $H$. 
\begin{theorem}[\cite{MS}, Theorem 1.5]\label{thm:morrissaxton}For every $\ell\geq 2$, there exist  $Q >0$, $\delta >0$ and $k_0\in \mathbb{N}$ such that   the following holds for every $k\geq k_0$ and every $n\in \mathbb{N}$. Given a graph $G$ with $n$ vertices and $kn^{1+1/\ell}$ edges there exists a collection $\mathcal{F}$  of copies of $C_{2\ell}$ in $G$ satisfying 

\begin{itemize}
\item[(a)]  $|\mathcal{F}|\geq \delta k^{2\ell}n^2,$
\item[(b)]  $d_{\mathcal{F}}(\sigma)\leq Q  k^{2\ell-j-\frac{j-1}{\ell-1}}n^{1-\frac{1}{\ell}}$ for every $j$-tuple $\sigma$ for $1\leq j\leq 2\ell-1$. 
\end{itemize}
\end{theorem}
\end{section}
\begin{section}{Proof}\label{proofs}
The main novelty of this paper is in proving a new ``balanced supersaturation" result for linear even $r$-uniform cycles. Such a result   states roughly that an $r$-graph $G$ on $n$ vertices with significantly more than
$\ex(n, C^{(r)}_{2\ell})$  edges contains many copies of  $C_{2\ell}^{(r)}$ which are additionally distributed relatively uniformly over the edges of $G$. This was already proved in~\cite{BNS,FMS} but it was not strong enough to establish a corresponding result in random graphs. The output of our supersaturation result depends crucially on the codegrees  of the underlying $r$-graph in addition to the density. We first pass to a large $r$-partite subgraph of our $r$-graph. Then using a procedure that we call regularization, we make sure that for each pair of partition classes, $V_i,V_j$ all codegrees are within a constant factor from each other (the codegree is called $\Delta_{ij}$ in the proof). We then show that there must exist a pair which has  a ``large'' $2$-shadow, as otherwise the original $r$-graph will not have $Kn^{r-1}$ edges (see Claim~\ref{keyclaim}). Note that for $r=3$, this bound is not large enough for the rest of the arguments to provide optimal bounds, but if one can improve this, then better bounds for $r=3$ will immediately follow. Then we use a corresponding balanced supersaturation result of graphs from~\cite{MS} to obtain a collection of graph cycles in this $2$-shadow graph (without loss of generality, it is assumed the largest $2$-shadow is $\partial_{12}$, the one coming from the pair $V_1, V_2$) which are distributed relatively uniformly over the edges of the $2$-shadow graph. Now because of the regularization, we are able to obtain the desired collection of hypergraph cycles in the original hypergraph by extending the graph cycles in the $2$-shadow graph to hypergraph cycles in a natural way.

It is crucial for our method to keep the codegree  dependence along the way.  We do not have any non-trivial upper or lower bounds on $\Delta_{12}$, the codegree of the pairs of vertices in the shadow graph $\partial_{12}$, in particular, it can be as large as $n$. However, by keeping $\Delta_{12}$  as a parameter in the balanced supersaturation statement and using the fact that the overall edge density in the $r$-graph is $K n^{r-1}$, we are able to quantify the tradeoff between the codegree $\Delta_{12}$ and the size of the largest $2$-shadow, $\partial_{12}$.

In summary, by keeping track of the codegree in the original $r$-graph while working in the $2$-shadow graph we obtain non-trivial results for  $\ex\left(G_{n,p}^{(r)}, C_{2\ell}^{(r)}\right)$;   previous balanced supersaturation results~\cite{BNS, FMS} which exclusively worked in the $2$-shadow graph gave no non-trivial results unless $p$ is a constant. Due to this dependence, our supersaturation result appears as part of the proof of the following theorem (see (P5) and (P6)). Indeed, it would be repetitive to separate it from the rest of the proof as we need to keep the parameter $\Delta_{12}$ in the balanced supersaturation result.

\begin{theorem}\label{thm:onestep}For every  $r\geq 4$, $\ell\geq 2$, there exist $K_0,n_0\in \mathbb{N}$ and $\eps >0$ such that the following holds for all  $n\geq n_0$ and every $K\geq K_0(\log{n})^{2r(r-1)}$. Given an $r$-graph $G$ with $n$ vertices and $Kn^{r-1}$  edges there exists a collection  $\mathcal{C}$ of at most 
 \begin{equation} \label{eqn:numcon}\exp \left(\frac{\plog_{\ell,r}(n)}{\eps} n^{1+\frac{1}{2\ell-1}} \right)\end{equation}
subgraphs of $G$ such that 

\begin{itemize} \item [(i)] every $C_{2\ell}^{(r)}$-free subgraph of $G$ is a subgraph of some $C\in \mathcal{C}$, 
 \item [(ii)]  $e(C)\leq  \left(1-\frac{\eps}{(\log{n})^{r^2(\ell+1) }}\right)e(G)$, for each $C\in \mathcal{C}$.
\end{itemize}
\end{theorem}

\begin{proof} The proof has two parts. In the first part we find a large collection of $2\ell$-cycles in $G$ such that the degrees of $j$-tuples in this collection, for all $1\leq j\leq 2\ell-1$, are small enough so that in the second part we are able to apply Theorem~\ref{thm:MSContainers}. Thus, we obtain  containers which contain all $C_{2\ell}$-free subgraphs of $G$. Finally, we show that the containers satisfy (i) and (ii). 

For the first part we pass  to a large $r$-partite subgraph $H'$ of $G$ with  vertex partition $(U_1,U_2, \dots U_r)$ in which all the pairs of vertices lying in different partition classes that are in some edge together, are ``almost regular''. That is, for every $1\leq i < j\leq r $ there exists some number $\Delta_{ij}$ such that if we pick $u_i\in U_i,u_j\in U_j$  such that $u_i$ and $u_j$ are in some edge together then $\Delta_{i,j}\leq d_{H'}(u_i,u_j)\leq 2\Delta_{i,j}$ (up to polylog factors). Moreover, we can guarantee that no vertices in $H'$ are isolated.  Then we find a pair of partition classes, say $U_1,U_2$, with respect to which the shadow is large, that is, $\partial_{U_1U_2}(H')$ is a dense enough  $2$-graph to apply Theorem~\ref{thm:morrissaxton} and obtain a large collection of $2\ell$-cycles in $\partial_{U_1U_2}(H')$. Then we  expand these cycles  using the regularity of pairs of vertices in $H'$ to $2\ell$-cycles in $H'$, and thus in $G$. The  second part of the proof is more technical.  We show that for all $1\leq j\leq 2\ell-1$, the $j$-tuples behave well enough so that the assumptions of  Theorem~\ref{thm:MSContainers} are satisfied.
\bigskip

{\bf Part 1. Balanced supersaturation}

Let  $Q, \delta_0, k_0$ be obtained from Theorem~\ref{thm:morrissaxton} applied with $\ell$. Let $\eps_0$ be obtained from Theorem~\ref{thm:MSContainers} applied with $r_{\ref{thm:MSContainers}}=2\ell$. Choose constants

$$R={r\choose 2}\qquad  \alpha_r = \frac{r!}{2r^{R+r}} \qquad  \beta_r = \frac{\alpha_r}{4Rr^R} \qquad K_0=8 \,k_0^2 \, \beta_r^{-2} \qquad \delta\ll \min\{\eps_0, \, \delta_0\}\qquad \eps=\delta^4.$$ 

By a classical result of Erd\H{o}s-Kleitman~\cite{EK}, $G$ has an $r$-partite subgraph $H$ with 
$r$-partition $(V_1,V_2,\dots, V_r)$
such that $e(H)\geq r!e(G)/r^r$. Let $$D=\{\textbf{s}= \{s_{i,j}\}_{1\leq i<j\leq r } \,: s_{i,j} \in \{1, 2, \ldots,  \lfloor (r-2)\log{n} \rfloor\}\}.$$ 
For each $\textbf{s}=\{s_{i,j}\}_{1\leq i<j\leq r }\in D$, let
$$E(\textbf{s}) = \{ \{v_1,v_2,\dots, v_r\}\in E(H): v_i\in V_i \hbox{ and }
2^{s_{i,j}}\leq d_H(v_i,v_{j})<2^{s_{i,j}+1} \hbox{ for all $i,j$}\}.$$ Since $\cup_{\textbf{s}}{E(\textbf{s})}$ is a partition of  $E(H)$,   by the pigeonhole principle there exists  $\textbf{s}_0 \in D$ such that $$|E(\textbf{s}_0)|\geq \frac{e(H)}{|D|} > \frac{e(H)}{(r\log{n})^R}.$$  For every $1\leq i <j\leq r$ and $s_{i,j}\in \textbf{s}_0$, let  $\Delta_{ij}=2^{s_{i,j}}$. Let $H_0$ be the 
subgraph with   $$
 V(H_0) =\bigcup_{e \in E(\textbf{s}_0)} \, e \qquad \hbox{and}\qquad
 E(H_0) = E(\textbf{s}_0).$$ Let $U_i=V_i \cap V(H_0)$. By definition, 
\begin{equation}\label{eq:edgesofH}e(H_0)\geq \frac{e(H)}{(r\log{n})^R} \geq \frac{r!e(G)}{r^{R+r}(\log{n})^R}\geq  \frac{r!Kn^{r-1}}{r^{R+r}(\log{n})^R}.\end{equation}

Now construct a sequence of $r$-graphs $H_0\supset H_1\supset \cdots \supset H_m=H'$ as follows. For each $q\geq 0$ and each pair $v_i\in U_i,v_j\in U_j$ if $$d_{H_q}(v_i,v_j)< \frac{d_{H}(v_i,v_j)}{2R(r\log{n})^R},$$ 
then delete all edges containing the pair $v_i,v_j$ from $H_q$ and let the resulting $r$-graph be $H_{q+1}$.  By (\ref{eq:edgesofH}), the number of edges deleted during this process is at most 
$$\sum_{1\leq i<j\leq r}\sum_{v_i\in V_i, v_j\in V_j}{\frac{d_H(v_i,v_j)}{2R(r\log{n})^R}} \leq \frac{e(H)}{2(r\log{n})^R} \leq \frac{e(H_0)}{2}.$$
Consequently, the process terminates with $H'$ and $e(H') \ge e(H_0)/2$.  Again using (\ref{eq:edgesofH}) we obtain
\begin{equation}\label{eq:edgesofHprime}
e(H')\geq  \frac{e(H_0)}{2} \ge   \frac{r!Kn^{r-1}}{2r^{R+r}(\log{n})^R} =  \frac{\alpha_r  Kn^{r-1}}{(\log{n})^R}.
\end{equation}
We also delete all the vertices which became isolated in $U_i$ for all $1\leq i \leq r$. 
For simplicity of presentation,  we keep using the letters $U_i$ for the parts of $H'$.  Let $\partial_{i,j}=\partial_{U_i,U_j}(H')$. Now $H'$ satisfies the following, for all $1\leq i<j\leq r$:
\begin{itemize}
\item  [(P1)] there are no isolated vertices in $H'$
\item  [(P2)] if  $v_i\in U_i, v_j\in U_j$  such that there exists $e\in E(H')$ with $\{v_iv_j\}\subseteq e$  then  $$\frac{\Delta_{ij}}{2R(r\log{n})^R}\leq d_{H'}(v_i,v_j) \leq 2\Delta_{ij}.$$
\end{itemize}

Without loss of generality, we may assume that $|U_1|\geq |U_2|\geq \dots |U_r|$.

\begin{claim} \label{keyclaim} There exists some $j\in\{2,\dots, r\}$ such that $$|\partial_{1j}| \geq \frac{K^{\frac{1}{r-1}} |U_1|^{2-\frac{1}{r-1}}}{\plog_{\ell,r}(n)}.$$
\end{claim}
\begin{proof} We may assume $j=2$ does not satisfy the desired bound, as otherwise there is nothing to prove. We will use this to show that there exists some $j\in \{3,\dots, r\}$ which is as large as needed.

By (\ref{eq:edgesofHprime}), (P2) and the definition of $\partial_{12}$ we have:

\begin{equation} \frac{\alpha_r  Kn^{r-1}}{(\log{n})^R} \leq e(H') \leq  |\partial_{12}| (2\Delta_{12}).
\end{equation}

Hence,  \begin{equation}\label{bound:codeg}|\Delta_{12}| \geq  \frac{\alpha_r  Kn^{r-1}}{2|\partial_{12}|(\log{n})^R}.\end{equation} 

Now fix any $u_1\in U_1$. Since there are no isolated vertices in $H'$, there exists $u_2\in U_2$ such that $d_{H'}(u_1,u_2)>0$. By (P2) and~(\ref{bound:codeg}), we obtain:

$$d_{H'}(u_1,u_2)\geq \frac{\Delta_{12}}{2R(r\log{n})^R} \geq \frac{\alpha_r  Kn^{r-1}}{4Rr^R(\log{n})^{2R}|\partial_{12}|}.$$

Recall that $u_1$, $u_2$ are arbitrary vertices in $U_1$ and $U_2$ which are contained in at least one edge. Each of these edges is of form $\{u_1, u_2, u_3, \dots, u_r\}$, where $u_j\in U_j$, for $j=3, \dots, r$. Now some of these $u_j$'s may be repeated, but by thepigeonhole principle there exists $j$ such that for at least $d_{H'}(u_1,u_2)^{1/(r-2)}$   distinct $x\in U_j$, the set  $\{u_1, u_2, x\}$ is contained in some edge. So  for this fixed $u_1\in U_1$ we obtained some $j=j(u_1)\in \{3, \dots r\}$ and a set $X_{j,u_1}\subseteq U_j$ such that $u_1,x_j$ are contained in an edge for each $x_j\in X_{j,u_1}$, and 

$$|X_{j,u_1}|\geq d_{H'}(u_1,u_2)^{1/(r-2)}\geq \left(\frac{\alpha_r  Kn^{r-1}}{4Rr^R(\log{n})^{2R}|\partial_{12}|}\right)^{1/(r-2)}=\left(\frac{ Kn^{r-1}}{\plog_{\ell,r}(n)|\partial_{12}|}\right)^{1/(r-2)}.$$

Once again by the pigeonhole principle, for at least $|U_1|/r$  vertices $u_1\in U_1$, $j(u_1)=j$ is the same. Note that for each $u_1\in U_1$ and $x\in X_{j,u_1}$, we obtain distinct edges $u_1x\in \partial_{1j}$. Clearly $|U_1|>K/\plog_{\ell,r}(n)$ for otherwise 
$e(H')\le |U_1|n^{r-1}< Kn^{r-1}/\plog_{\ell, r}(n)$. Using this and $\partial_{12}\le |U_1|^2$, 

$$|\partial_{1j}|\geq |U_1|\left(\frac{ Kn^{r-1}}{\plog_{\ell,r}(n)|\partial_{12}|}\right)^{1/(r-2)} \geq   |U_1|\left(\frac{ K |U_1|^{r-1}}{\plog_{\ell,r}(n)|U_1|^{2-\frac{1}{r-1}} K^{\frac{1}{r-1}}}\right)^{1/(r-2)} = \frac{K^{\frac{1}{r-1}} |U_1|^{2-\frac{1}{r-1}}}{\plog_{\ell,r}(n)},$$
as we wanted to show.
\end{proof}

By Claim~\ref{keyclaim}, we may assume  without loss of generality that

\begin{equation}\label{eq:partialedges}|\partial_{12}| \geq  \frac{K^{\frac{1}{r-1}} |U_1|^{2-\frac{1}{r-1}}}{\plog_{\ell,r}(n)} \geq \frac{K^{\frac{1}{r-1}} m^{2-\frac{1}{r-1}}}{\plog_{\ell,r}(n)},
\end{equation} 
where $m=|U_1|+|U_2|$.

\begin{claim}\label{claim:degrees}  $\Delta_{12}\geq 8\ell r^{R+1} (\log{n})^R n^{r-3}$. 
\end{claim}
\begin{proof}
By $(\ref{eq:edgesofHprime})$ and (P2),
$$ \frac{\alpha_r  Kn^{r-1}}{(\log{n})^R }\leq e(H')\leq |\partial_{12}| (2\Delta_{12}) \leq 2n^2 \Delta_{12}.$$ 
Also $K \geq K_0 (\log n)^{4R} \ge 8k_0^2\beta_r^{-2} (\log{n})^{4R}$. Hence  
$$ \Delta_{12} \geq \frac{\alpha_r  Kn^{r-3}}{2(\log{n})^R}  
\ge \frac{\alpha_r  8k_0^2\beta_r^{-2} (\log{n})^{4R}n^{r-3}}{2(\log{n})^R} =
   4\alpha_r \beta_r^{-2} k_0^2 (\log{n})^{3R}n^{r-3}
\ge 
8\ell R r^{R+1} (\log{n})^R n^{r-3},$$
where the last inequality follows since $n$ is sufficiently large compared to $\ell, r, k_0$.
\end{proof}

Set $k=\frac{|\partial_{12}|}{m^{1+1/\ell}}$. By (\ref{eq:partialedges}) and $K \geq 8k_0^2\beta_r^{-2} (\log{n})^{2r(r-1)}$, we can guarantee that 
$k \geq k_0,$  thus by Theorem~\ref{thm:morrissaxton} applied with $k$ and $m$, the  shadow graph $\partial_{12}$ contains a collection $\mathcal{F}$  of copies of $C_{2\ell}$ satisfying 

\begin{itemize}
\item[(P3)]  $|\mathcal{F}|\geq \delta k^{2\ell}m^2 ,$
\item[(P4)] $\Delta_j(\mathcal{F})\leq Q  k^{2l-j-\frac{j-1}{\ell-1}}m^{1-\frac{1}{\ell}}$ for all $1\leq j\leq 2\ell-1$. 
\end{itemize}

Let $C$ be any $2\ell$-cycle included in  $\mathcal{F}$ with consecutive edges $x_1x_2\dots x_{2\ell}$ in the natural cyclic ordering. Since  $d_{H'}(x_i,x_{i+1})\geq \Delta_{12}/2R(r\log{n})^R$ for every $1\leq i <j \leq \ell$, it follows  by (P2) that the number of ways to extend $C$ to some linear $2\ell$-cycle in $H'$ is at least
 $$\frac{\Delta_{12}}{2R(r\log{n})^R}\left(\frac{\Delta_{12}}{2R(r\log{n})^R}-(r-2)n^{r-3}\right)\dots \left(\frac{\Delta_{12}}{2R(r\log{n})^R}-(2\ell-1)(r-2)n^{r-3}\right).$$
By Claim~\ref{claim:degrees}, this is at least  $$\left(\frac{\Delta_{12}}{4R(r\log{n})^R}\right)^{2\ell} =\left(\frac{\Delta_{12}}{\plog_{\ell,r}(n)}\right)^{2\ell} . $$
Let $\Ext(C)$ be the collection of all cycles $C_{2\ell}^{(r)}$ obtained from $C$ in this manner. Let $\mathcal{F}' =\{\Ext(C)|C\in \mathcal{F}\}$, so $\mathcal{F}'$ is collection of edge sets of linear $C_{2\ell}^{(r)}$ in $H'$ (and, as in Theorem~\ref{thm:morrissaxton}, $\mathcal{F}'$ can also be viewed as a $2\ell$-graph). By the previous discussion and (P3) and (P4):
\begin{itemize}
\item[(P5)]  $|\mathcal{F}'|\geq \delta k^{2l}m^2 \left(\frac{\Delta_{12}}{\plog_{\ell,r}(n)}\right)^{2\ell},$
\item[(P6)]  $\Delta_j(\mathcal{F}') \leq Q  k^{2\ell-j-\frac{j-1}{\ell-1}}m^{1-\frac{1}{\ell} } (2\Delta_{12})^{2\ell -j}$ for  all $1\leq j\leq 2l-1$.
\end{itemize}
\bigskip

{\bf Part 2. Containers}

Lt ${S}$ be the $2\ell$-graph with $$V(S)=E(G) \qquad \hbox{  and } \qquad E(S) = {E(G) \choose 2\ell} \cap \mathcal{F}'.
$$
In other words, vertices of $S$ are edges of $G$ and edges of $S$ are copies of $C_{2\ell}^{(r)}$  in $\mathcal{F}'$.  The edges in $E(G)\setminus E(H')$ are isolated vertices in ${S}$, and $ \Delta_j(S)\leq \Delta_j(\mathcal{F}') $. Therefore, (P5) and (P6) translate to the properties of $e(S)=|\mathcal{F}'|$ and    $\Delta_j(S) = \Delta_j(\mathcal{F}')$, correspondingly.

\begin{claim}\label{claim:partialdegrees}$  |\partial_{12}| \Delta_{12} \geq \frac{\alpha_r  Kn^{r-1}}{2(\log{n})^R} $.
\end{claim}
\begin{proof} By (\ref{eq:edgesofHprime}) and (P2), 
$$\frac{\alpha_r  Kn^{r-1}}{(\log{n})^R} \leq e(H') \leq  |\partial_{12}| (2\Delta_{12}),$$ and the claim follows.
\end{proof}

Now let us compute the co-degree  function $\delta(S,\tau)$, for the following choice of $\tau$: 
\begin{equation}\label{eq:tau}\tau=\delta^{-3}\plog_{\ell,r}(n) K^{-1} n^{-(r-2)+\frac{1}{2\ell-1}} .\end{equation}

By the definitions of $H_0$ and $H'$,
$$e(G)\leq \frac{r^{R+r}(\log{n})^R}{r!}e(H_0) \le
\frac{2r^{R+r}(\log{n})^R}{r!}e(H')
\leq \plog_{\ell,r}(n) |\partial_{12}| \Delta_{12},$$

and it follows that 
\begin{equation}\label{eq:avgdegreeS}d(S)=\frac{2\ell e(S)}{e(G)} \geq \frac{ 2\delta \ell k^{2\ell}m^2  \left(\frac{\Delta_{12}}{\plog_{\ell,r}(n)}\right)^{2\ell} }{\plog_{\ell,r}(n) |\partial_{12}|\Delta_{12}} \geq \frac{\delta k^{2\ell-1} \Delta_{12}^{2\ell-1}m^{1-1/\ell} }{\plog_{\ell,r}(n)}.
\end{equation}

where in the  second inequality we used $|\partial_{12}|=km^{1+1/\ell}$ and that $n$ is sufficiently large.

By (\ref{eq:avgdegreeS}) and (P6)  for all $2\leq j \leq 2\ell -1$, we have 

\begin{align*}\left(\frac{\Delta_j(S)}{d(S)}\right)^{1/(j-1)}\tau^{-1}&\leq   \left(\frac{Qk^{2\ell-j-\frac{j-1}{\ell-1}} (2\Delta_{12})^{2\ell -j}\plog_{\ell,r}(n)}{\delta k^{2\ell-1} \Delta_{12}^{2\ell-1}}\right)^{1/(j-1)}\tau^{-1}\\
&\leq  \left(\delta^{-1}\plog_{\ell,r}(n) k^{-(j-1)-\frac{j-1}{\ell-1}} (\Delta_{12})^{-(j-1)}\right)^{1/(j-1)}\tau^{-1}\\
&\leq  \delta^{-1}\plog_{\ell,r}(n) k^{-1-\frac{1}{\ell-1}} (\Delta_{12})^{-1} \tau^{-1} \\
&\leq  \delta^{-1} \plog_{\ell,r}(n) \left(\frac{|\partial_{12}|}{m^{\frac{\ell+1}{\ell}}}\right)^{-\frac{\ell}{\ell-1}}\left(\frac{2|\partial_{12}|}{e(H')}\right)\tau^{-1} \\
&\leq  \delta^{-1}\tau^{-1} \plog_{\ell,r}(n) \frac{m^{1+\frac{2}{\ell-1}}|\partial_{12}|^{-\frac{1}{\ell-1}}}{Kn^{r-1}}\\
&\leq  \delta^{-1}\tau^{-1}\plog_{\ell,r}(n) \times \frac{m^{1+\frac{1}{(r-1)(\ell-1)}}}{K^{1+\frac{1}{(r-1)(\ell-1)}}n^{r-1}} \\
&\leq \delta^{-1}\tau^{-1} \plog_{\ell,r}(n) K^{-1-\frac{1}{(r-1)(\ell-1)}}n^{-(r-2)+\frac{1}{(r-1)(\ell-1)}} \\
&<\frac{\delta}{2\ell}.
\end{align*}

Above in the first equality we used  that $|\partial_{12}|=km^{1+1/\ell}$, in the fourth inequality we used (\ref{eq:partialedges}) and  Claim~\ref{claim:partialdegrees},  and in the last inequality we used the definition of $\tau$ and $\delta <1/2\ell$.

As for $j=2\ell$, since $\Delta_{2\ell}\leq 1$, it follows
\begin{align*}\left(\frac{\Delta_{2\ell}(S)}{d(S)}\right)^{1/(2\ell-1)} \tau^{-1}&\leq \left(\frac{\plog_{\ell,r}(n)}{\delta k^{2\ell-1} \Delta_{12}^{2\ell-1}m^{1-1/\ell} }\right)^{1/(2\ell-1)}\tau^{-1} \\
&\leq \tau^{-1}\delta^{-1} \frac{\plog_{\ell,r}(n) m^{1+\frac{1}{\ell}-\frac{\ell-1}{\ell(2\ell-1)}}}{|\partial_{12}|\Delta_{12}}\\
&\leq \tau^{-1}\delta^{-1} \plog_{\ell,r}(n) K^{-1} n^{-(r-2)+\frac{1}{2\ell-1}}\\
&<\frac{\delta}{2\ell},
\end{align*}
where in the second inequality we used  $|\partial_{12}|=km^{1+1/\ell}$ and in the third inequality we used   Claim~\ref{claim:partialdegrees}. So for all $2\leq j \leq 2\ell $,

\begin{equation}\label{ineq:degrees} \left(\frac{\Delta_j(S)}{d(S)}\right)\tau^{-(j-1)}<\left(\frac{\delta}{2\ell}\right)^{j-1}\leq \left(\frac{\delta}{2\ell}\right).\end{equation} 
By (\ref{ineq:degrees}),
$$\delta(S,\tau) = \frac{1}{d(S)}\sum_{j=2}^{2\ell}{\frac{\Delta_{j}(S)}{\tau^{j-1}}}<\delta,$$ and we can therefore apply  Theorem~\ref{thm:MSContainers} to $S$ and obtain a collection of $\mathcal{C}\subseteq  \mathcal{P}(V(S)) = \mathcal{P}(E(G))  $ of at most $$\exp\left(\frac{\tau |V(S)|\log{1/\tau}}{\delta}\right)$$  subsets of $V(S)$ such that

\begin{itemize}
\item [(P7)] for each independent set $I$ in  $S$ there exists a ``container'' $C\in \mathcal{C}$ such that  $I\subseteq C$,
\item [(P8)] $e(S[C])\leq (1-\delta) e(S)$, for each container $C\in \mathcal{C}$.
\end{itemize}

By the choice of $\tau$, and recalling that $\varepsilon=\delta^4$ and $|V(S)|=e(G)=Kn^{r-1}$ we get that
$$\left(\frac{\tau |V(S)|\log{1/\tau}}{\delta}\right) \leq  \exp \left(\frac{\plog_{\ell,r}(n)}{\eps} \max \left \{ n^{1+\frac{1}{(r-1)(\ell-1)}}K^{-\frac{1}{(r-1)(\ell-1)}}, \, n^{1+\frac{1}{2\ell-1}} \right\}\right)$$
 and this proves (\ref{eqn:numcon}). We will now prove statements (i) and (ii) of the theorem.

Each $C_{2\ell}^{(r)}$-free subgraph of $G$ is an  independent set  of $S$, and each $C\in \mathcal{C}$ can be viewed as  a subgraph of $G$. So (P7) implies that for each $C_{2\ell}^{(r)}$-free subgraph $G'$ of $G$  there exists a container $C\in \mathcal{C}$ such that  $G'\subseteq C$. This shows that (i) is true. To conclude, it remains to show that (P8) implies that statement (ii) of the theorem holds.

Let   $C\in \mathcal{C}$. Recall that by definition, every vertex $v\in V(S)$ is an edge in $E(G)$. For each $e\in E(S-S[C])$ there exists  $v\in e$ such that $v\in V(S)\setminus C= E(G)\setminus E(C)$. Since  $d_{S}(v)\leq \Delta_1(S)$ for every $v \in V(S)$,
\begin{equation}\label{eq:containers} e(S-S[C]) \leq \Delta_1(S) e(G-C).
\end{equation}

On the other hand,  (P5) and (P6) and $\delta$ sufficiently small imply that
\begin{align*}
\frac{e(S)}{\Delta_1(S)} &\geq \frac{\delta k^{2l}m^2  \left(\frac{\Delta_{12}}{4R(r\log{n})^R}\right)^{2\ell}}{Q  k^{2\ell-1}m^{1-\frac{1}{\ell}}  (2\Delta_{12})^{2\ell -1}}\\
&\geq  \frac{2}{Q(8Rr^{R})^{2\ell} } \cdot \frac{\delta km^{1+\frac{1}{\ell}}\Delta_{12}}{(\log{n})^{2\ell R}} \\
&\geq  \frac{\delta^2 km^{1+\frac{1}{\ell}}\Delta_{12}}{(\log{n})^{2\ell R}}.
\end{align*}
 Thus, together with (\ref{eq:containers})  and (P8) it follows that  $$\delta \, e(S) \leq e(S- S[C])\leq \Delta_1(S) \, e(G-C)\leq e(S) \,  \frac{(\log{n})^{2\ell R}}{\delta^2 km^{1+\frac{1}{\ell}}\Delta_{12}} \, e(G-C).$$
Therefore, 

$$e(G-C) \geq \frac{\delta^3 k\Delta_{12} m^{1+\frac{1}{\ell}}}{ (\log{n})^{2\ell R}}\geq  \frac{ \delta^3 (|\partial_{12}| \Delta_{12})}{ (\log{n})^{2\ell R}} \geq \frac{\delta^3 \alpha_r Kn^{r-1}}{2(\log{n})^{(2\ell+1) R}}\geq \frac{\delta^4 e(G)}{(\log{n})^{(2\ell+1) R}}\geq  \frac{\eps \, e(G)}{(\log{n})^{r^2(\ell+1)}} ,$$   where in the second inequality we used the definition of $k$, in the third one we used Claim~\ref{claim:partialdegrees}, in the fourth that $\delta$ is sufficiently small, and the last one we use $\eps=\delta^4$. As $e(C) = e(G)-e(G-C)$, the proof of (ii) is complete. \end{proof}

Let us restate the result that we have just proved.

 \begin{theorem}\label{thm:onestepwhenris4}For every  $r\geq 4$, $\ell\geq 2$, there exist $K_0,n_0\in \mathbb{N}$ and $\eps >0$ such that the following holds for all  $n\geq n_0$ and every $K\geq K_0(\log{n})^{2r^2}$. Given an $r$-graph $G$ with $n$ vertices and $Kn^{r-1}$  edges there exists a collection  $\mathcal{C}$ of at most 
$$\exp\left(\frac{1}{\eps}n^{1+\frac{1}{2\ell-1}} \plog_{\ell,r}(n)  \right)$$
subgraphs of $G$ such that 
\begin{itemize} \item [(i)] every $C_{2l}^{(r)}$-free subgraph of $G$ is a subgraph of some $C\in \mathcal{C}$, 
 \item [(ii)]  $e(C)\leq  \left(1-\frac{\eps}{\plog_{\ell,r}(n) }\right)e(G)$, for each $C\in \mathcal{C}$.
\end{itemize}
\end{theorem}

By iterating this theorem starting with $G=K_{n}^{(r)}$, the complete $r$-graph on $n$ vertices we can show that any $r$-graph on $n$ vertices which is $C_{2\ell}^{(r)}$-free is contained in one of the ``containers'' and there are not many of them.

\begin{theorem}\label{thm:countforr>=4} For every $r\geq 4$, $\ell\geq 2$ there exists  $K_0\in \mathbb{N}$ such that for all sufficiently large $n\in \mathbb{N}$ and every $K\geq K_0(\log{n})^{2r^2}$  there exists a collection $\mathcal{G}_{\ell,r}(n,K)$ of at most 

$$\mathrm{exp}\left(n^{1+\frac{1}{2\ell-1}} \plog_{\ell,r}(n)  \right)$$
 $r$-graphs on vertex set $[n]$ such that  $e(G)\leq Kn^{r-1}$ for each $G\in \mathcal{G}_{\ell,r}(n,K)$,  and every $C_{2\ell}^{(r)}$-free $r$-graph on $[n]$ is a subgraph of some $G\in \mathcal{G}_{\ell,r}(n,K)$. 
\end{theorem}

\begin{proof} Let $K_0$ and $\eps$ be obtained from Theorem~\ref{thm:onestepwhenris4} applied with $\ell$. We apply Theorem~\ref{thm:onestep} iteratively, each time refining the set of
containers obtained at the previous step.   We start with $\mathcal{C}_0=\left\{K_n^{(r)}\right\}$. For $i\geq 1$, let $$K_{i}=\max\left\{\left(1-\frac{\eps}{\plog_{\ell, r}(n)}\right)^in,K_0(\log{n})^{2r^2}\right\}.$$ At step $t$  we obtain a collection $\mathcal{C}_t$ of $r$-graphs on $[n]$ such that $e(G)\leq K_tn^{r-1}$ for every $G\in \mathcal{C}_t$  and every $C_{2\ell}^{(r)}$-free graph on $[n]$ is a subgraph of some $G\in \mathcal{C}_t$, and moreover,
\begin{equation}\label{count:total}|\mathcal{C}_t|\leq \text{exp}\left( \sum_{i=1}^{t} \frac{1}{\eps} n^{1+\frac{1}{2\ell-1}}\plog_{\ell, r}(n) \right).
\end{equation}

For $i\geq 0$, at  step $i+1$, we apply Theorem~\ref{thm:onestepwhenris4} to each graph $G\in \mathcal{C}_i$ with $e(G)\geq K_{i+1}n^{r-1}$  (note that  $e(G)\leq K_{i}n^{r-1}$, as otherwise $G$ would not have been in $ \mathcal{C}_i$) and obtain a family  $\mathcal{C}(G)$ of subgraphs of $G$ of with the following properties:
\begin{itemize} \item [(a)] every $C_{2l}^{(r)}$-free subgraph of $G$ is a subgraph of some $C\in \mathcal{C}(G)$, 
 \item [(b)]   For each $C\in \mathcal{C}(G)$, $$e(C)\leq \left(1-\frac{\eps}{\plog_{\ell, r}(n)}\right)e(G)\leq    \left(1-\frac{\eps}{\plog_{\ell, r}(n)}\right) K_in^{r-1} \leq K_{i+1}n^{r-1}.$$
\end{itemize}

If $e(G)\leq K_{i+1}n^{r-1}$  we let $\mathcal{C}(G)=\{G\}$. We define $\mathcal{C}_{i+1}=\cup_{G\in \mathcal{C}_i}{\mathcal{C}(G)}.$ Let $m$ be the minimum such that $K_m\leq K$.  We iterate until we obtain $\mathcal{C}_m$. It is easy to check that  $m=O(\log{n})$. This allows us to get the desired bound on the size of the $\mathcal{C}_m$. 
\begin{equation}\label{count:total}|\mathcal{C}_m|\leq \text{exp}\left(\frac{1}{\eps} \sum_{i=1}^{m} n^{1+\frac{1}{2\ell-1}} \plog_{\ell, r}(n) \right)\leq  \text{exp}\left(n^{1+\frac{1}{2\ell-1}} \plog_{\ell, r}(n)  \right).
\end{equation}

To finish the proof, we let $\mathcal{G}_{\ell,r}(n,K)=\mathcal{C}_m$.
\end{proof}

We are ready to state our main result.

\begin{theorem}\label{thm:main} Fix $r\geq 4$ and $\ell\geq 2$. 
Set $p_1= n^{-(r-2)+\frac{1}{2\ell-1}+o(1)}$ then a.a.s. for all $p\geq p_1$, $$\ex\left(G_{n,p}, C_{2\ell}^{(r)}\right)\leq  pn^{r-1}\plog_{r,\ell}(n).$$
\end{theorem}

\begin{proof}
Let $C$, $K_0$ and $n_0$  be derived from Theorem~\ref{thm:countforr>=4} such that the statement holds. Now let us assume  $p\geq p_1$ and let  $K=K_0(\log{n})^{2r^2}$. By Theorem~\ref{thm:countforr>=4} there exists a collection $\mathcal{G}_{\ell,r}(n,K)$ of $r$-graphs on   vertex set $[n]$ such that $e(G)\leq Kn^{r-1}$ for every $G\in \mathcal{G}_{\ell,r}(n,K)$ and such that  every $C_{2\ell}^{(r)}$-free $r$-graph on $[n]$ is a subgraph of some $G\in \mathcal{G}_{\ell, r}(n,K)$. Furthermore,
$$|\mathcal{G}_{\ell,r}(n,K)|\leq  \text{exp}\left(n^{1+\frac{1}{2\ell-1}} \plog_{\ell, r}(n) \right).$$

Set $m=pn^{r-1} \plog_{r,\ell}(n) $ and suppose that $G_{n,p}$ has a $C_{2\ell}^{(r)}$-free subgraph $H$ with $m$ edges. Then there exists some $G\in \mathcal{G}_{\ell,r}(n,K)$ which contains $H$, in particular $G$ contains at least $m$ edges of $G_{n,p}$. Therefore, 
the expected number of such subgraphs $H$ is at most
$$ |\mathcal{G}_{\ell,r }(n,K)| {Kn^{r-1} \choose m} p^m.$$ 
The choice of $K$ and $p >p_1$ imply that
$$m \gg \max \{\log(|\mathcal{G}_{\ell, r}(n, K))|, \, pKn^{r-1} \}.$$ 
Consequently, 
\begin{align*}
    |\mathcal{G}_{\ell, r}(n,K)| {Kn^{r-1} \choose m} p^m \leq   \left(O(1) \cdot \frac{pK n^{r-1}}{m}\right)^m\longrightarrow 0,
\end{align*}
and the proof is complete.
\end{proof}

\subsection{Lower bounds}
\begin{proposition}\label{prop:lowerbound1} For every $\ell\geq 2, r\geq 3$, if $n^{-r}\ll p\ll  n^{-(r-1)+\frac{1}{2\ell-1}}$, then a.a.s. $G_{n,p}^{(r)}$ contains a $C_{2\ell}^r$-free subgraph with $(1-o(1))p{n \choose r}$  edges.
\end{proposition}

\begin{proof} Note that $\E[e(G_{n,p}^{(r)})]= p{n \choose r} $.  Let ${X}$ denote the number of copies of $C_{2\ell}^{(r)}$ in $G^{(r)}_{n,p}$. Then $$\E[{X}]=O(n^{2\ell(r-1)}p^{2\ell}).$$
Let   $\omega(n) = \frac{\E[e(G^{(r)}_{n,p})]}{\E[X]}$. By assumption on $p$, we have that $\omega(n)\rightarrow \infty$ as $n\rightarrow \infty$. Let $\eps(n)$ by any function such that $\eps(n)\rightarrow0$ as $n\rightarrow \infty$ but such that $\eps(n)\omega(n)\rightarrow \infty$. Then by Markov's inequality $$\Prob\left[X\geq \eps(n) p {n \choose r }\right]\leq \frac{1}{\eps(n) \omega(n) } \longrightarrow 0.$$  Thus, with high probability $X=o(  p{n \choose r})=o(\E[e(G_{n,p}^{(r)})])$. On the other hand, it is a well known fact that for $p\gg n^{-r}$, the random variable $e(G^{(r)}_{n,p})$ is a binomial variable concentrated around its mean, hence with high probability $X=o(e(G^{(r)}_{n,p}))$. Therefore, with high probability, by deleting one edge from each copy of  $C_{2\ell}^{(r)}$ we will obtain a $C_{2\ell}^{(r)}$-free subgraph of $G_{n,p}^{(r)}$ with $(1-o(1))e(G_{n,p}^{(r)})$  edges. 
\end{proof}

\begin{proposition}\label{prop:lowerbound2} For every $\ell\geq 2, r\geq 3$, if $p\gg n^{-(r-1)}$, then a.a.s. $G_{n,p}^{(r)}$ contains a $C_{2\ell}^{(r)}$-free subgraph with $(1-o(1))p{n-1 \choose r-1}$  edges.
\end{proposition}
\begin{proof} For any vertex $v$, we have $\E[d(v)]=p { n -1 \choose r-1}$ thus for $p\gg n^{-(r-1)}$ by Chernoff's inequality with high probability, $d(v)=(1+o(1))p { n -1 \choose r-1}$. Therefore, by letting $H$ be the subgraph of $G_{n,p}^{(r)}$ comprising all the edges containing a fixed  vertex $v$ we obtain a $C_{2\ell}^{(r)}$-free subgraph of the required size. 
\end{proof}
\end{section}

\begin{section}{Concluding remarks}
It remains an open problem to determine the full behaviour of  $\ex\left(G_{n,p}^{(r)}, C_{2\ell}^{(r)}\right)$ for $r=3$. The methods in this paper give the following result.
\begin{theorem}\label{thm:mainforr=3}For every $\ell\geq 2$, a.a.s. the following holds: 
$$ \ex\left(G_{n,p}^{(3)}, C_{2\ell}^{(3)}\right) \leq \begin{cases}  p^{\frac{1}{2\ell-1}}n^{1+\frac{2}{2\ell-1}+o(1)}, & \mbox{if } {n^{-1+o(1)}\leq p\leq n^{-1+\frac{1}{2\ell-2}+o(1)}} \\ 
pn^{2+o(1)},& \mbox{if }  p\geq n^{-1+\frac{1}{2\ell-2}+o(1)}.\end{cases}$$
\end{theorem}


\begin{figure}[H]
  \centering
    \includegraphics[width=.8\textwidth]{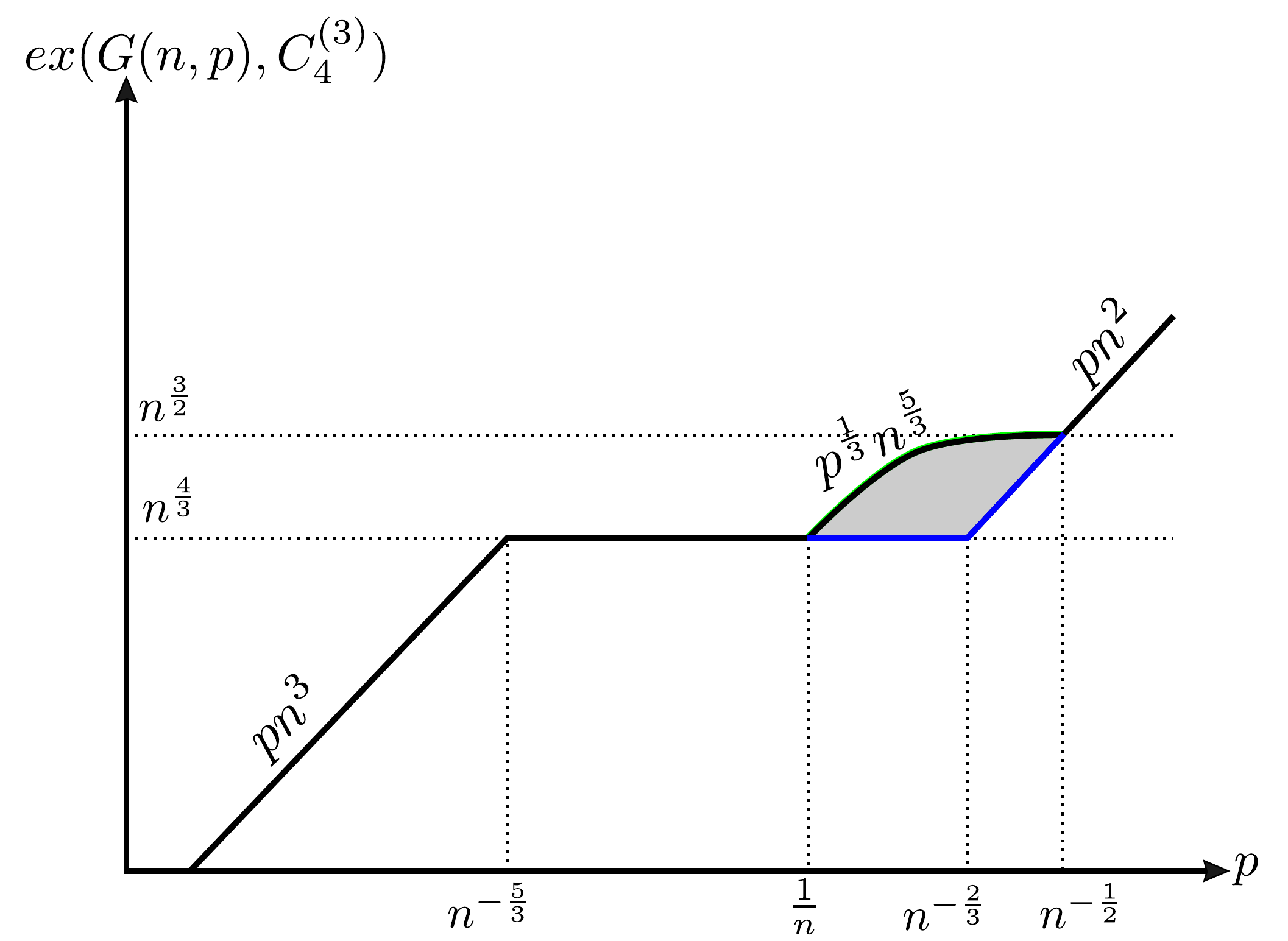}
  \caption{The behaviour of $\ex(G_{n,p}^{(3)}, C_4^{(3)})$ up to polylogarithmic factors}
  \label{c4}
\end{figure}
 Figure 2 shows that we do not know the optimal  behaviour of  $\ex(G_{n,p}^{(3)}, C_{4}^{(3)})$ in the regime $1/n\leq p\leq 1/\sqrt{n}$. We know of various constructions which reach the lower bound $n^{4/3}$ up to polylogarithmic factors  in the regime $n^{-1}\leq p\leq n^{-2/3}$ (see~\cite{MY} for details). This leads  us to conjecture that $n^{4/3}$ is the correct growth rate for this range of $p$.

\begin{conjecture}$$ \ex\left(G_{n,p}^{(3)}, C_{4}^{(3)}\right) = \begin{cases} (1+o(1))e(G_{n,p}^{(3)}), & \mbox{if } {n^{-3}\ll p \ll n^{-5/3}},\\ 
\Theta(n^{4/3+o(1)}), & \mbox{if } {n^{-5/3}\ll p \ll n^{-2/3} },\\
\Theta(pn^{2}),& \mbox{otherwise.} \end{cases}$$
\end{conjecture}

We refer the reader to~\cite{MY} for the full details in the case $r=3$.

\end{section}
\bigskip

{\bf Acknowledgments.} 
We wish to thank Jozsi Balogh, Wojtek Samotij and Rob Morris for several fruitful discussions about this problem. This work is based on the preprint~\cite{MY}. The main difference lies in an improved Claim~\ref{keyclaim}. We would like to thank Gwen McKinley for noticing that this claim was suboptimal and for sharing her ideas on how to improve   it for a certain range of $r$ and $\ell$, and for carefully reading  the preprint~\cite{MY}.  Recently we learned that Jiaxi Nie~\cite{Nie} obtained similar results for $r\geq 4$ as ours and also improved our bounds on $r=3$, $\ell\geq 2$. His result for $r\geq 4$ a priori employs  our methodology  from our previous preprint~\cite{MY}, that is, the regularization method and studying the largest $2$-shadow of the underlying $r$-graph, however, instead of improving Claim~\ref{keyclaim} as we do here, he  employs a further dihcotomic technique of either using $k$-shadows for $k>2$  or building the cycles in a greedy fashion similar to  the method of Balogh, Narayanan, Skokan~\cite{BNS}.

\end{document}